\title{On Local Super-Penalization of \\ Interior Penalty Discontinuous Galerkin Methods}

\date{\today}

\documentclass[11pt,leqno]{article}

\usepackage[paper=a4paper,dvips,top=3.0cm,left=3.0cm,right=3.0cm,
    foot=1cm,bottom=3.0cm]{geometry}\usepackage[dvips]{color}

\usepackage{authblk}

\author[1]{Andrea Cangiani\thanks{andrea.cangiani@le.ac.uk}}
\author[2]{John Chapman\thanks{john.chapman@durham.ac.uk}}
\author[1]{Emmanuil H. Georgoulis\thanks{emmanuil.georgoulis@le.ac.uk}}
\author[2]{Max Jensen\thanks{m.p.j.jensen@durham.ac.uk}}

\affil[1]{Department of Mathematics, University of Leicester, University Road, Leicester, United Kingdom}
\affil[2]{Department of Mathematics, University of Durham, Durham, United Kingdom}

\usepackage{graphicx}
\usepackage{subfigure}
\usepackage{psfrag}
\usepackage[dvips]{color}
\usepackage[usenames,dvipsnames]{xcolor}
\definecolor{ltgray}{gray}{0.95}

\usepackage[]{amsmath}

\usepackage{amsfonts}
\usepackage{amssymb}
\usepackage{textcomp}

\usepackage{amsthm}

\usepackage{amsopn}

\usepackage{xspace}

\DeclareMathAlphabet{\mathpzc}{OT1}{pzc}{m}{it}
\DeclareMathAlphabet{\mathcalligra}{T1}{calligra}{m}{n}

\usepackage{enumitem}

\usepackage{perpage}
\MakePerPage{footnote}

\definecolor{refkey}{rgb}{1,0,0}
\definecolor{labelkey}{rgb}{1,0,0}

\usepackage[pagewise]{lineno} 
\newcommand*\patchAmsMathEnvironmentForLineno[1]{%
  \expandafter\let\csname old#1\expandafter\endcsname\csname #1\endcsname
  \expandafter\let\csname oldend#1\expandafter\endcsname\csname end#1\endcsname
  \renewenvironment{#1}%
     {\linenomath\csname old#1\endcsname}%
     {\csname oldend#1\endcsname\endlinenomath}}%
\newcommand*\patchBothAmsMathEnvironmentsForLineno[1]{%
  \patchAmsMathEnvironmentForLineno{#1}%
  \patchAmsMathEnvironmentForLineno{#1*}}%
\AtBeginDocument{%
\patchBothAmsMathEnvironmentsForLineno{equation}%
\patchBothAmsMathEnvironmentsForLineno{align}%
\patchBothAmsMathEnvironmentsForLineno{flalign}%
\patchBothAmsMathEnvironmentsForLineno{alignat}%
\patchBothAmsMathEnvironmentsForLineno{gather}%
\patchBothAmsMathEnvironmentsForLineno{multline}%
}

\usepackage{listings}
\usepackage[font=small,labelfont=bf]{caption}

\newtheorem{theorem}[equation]{Theorem}
\newtheorem{definition}[equation]{Definition}
\newtheorem{lemma}[equation]{Lemma}

\numberwithin{equation}{section}
\numberwithin{figure}{section}
\numberwithin{table}{section}

\newcommand\CC{\Lang{\mbox{C++}}\xspace}

\newcommand\Lang[1]{\textsc{#1}}

\newcommand{\bdm}{\begin{displaymath}}
\newcommand{\edm}{\end{displaymath}}
\newcommand{\beq}{\begin{equation}}
\newcommand{\eeq}{\end{equation}}

\newcommand{\newVar}[2]{\newcommand{#1}{\ensuremath{#2}\xspace}}
\newVar\Naturals{\mathbb{N}}
\newVar\Integers{\mathbb{Z}}
\newVar\Rationals{\mathbb{Q}}
\newVar\Reals{\mathbb{R}}
\newVar\Complex{\mathbb{C}}
\newVar\Id{\mathbb{I}}

\newcommand{\code}[1]{\lstinline{#1}}

\newcommand{\norm}[1]{\ensuremath{\lVert#1\rVert}}

\newcommand{\triple}[1]{\ensuremath{|\hspace{-1pt}|\hspace{-1pt}| #1 |\hspace{-1pt}|\hspace{-1pt}|}}

\newcommand{\abs}[1]{\ensuremath{\lvert#1\rvert}}

\newcommand{\set}[1]{\ensuremath{\{#1\}}}

\newcommand{\Average}[1]{\ensuremath{\{ \! \! \{ #1 \}  \! \! \}} }
\newcommand{\average}[1]{\ensuremath{\Average{#1}}}

\usepackage{stmaryrd}
\newcommand{\Jump}[1]{\ensuremath{\llbracket #1 \rrbracket}}
\newcommand{\jump}[1]{\ensuremath{\Jump{#1}}}

\newcommand{\transpose}[0]{{\ensuremath{\top}}}
\newcommand{\ex}[0]{\ensuremath{\mathrm{e}}}
\newcommand{\Order}[1]{{\ensuremath{\mathcal{O}(#1)}}}

\newVar\pqTree{{\cal P}}
\newVar\groups{\mathbf{G}}
\newVar\lastGroup{\ell}
\newVar{\currentId}{\ensuremath{\id}\xspace}

\renewcommand\epsilon{\varepsilon}
\renewcommand\phi{\varphi}
\renewcommand\theta{\vartheta}

\newcommand{\vectau}[0]{{\ensuremath{\boldsymbol{\tau}}}}

\newcommand{\vecx}[0]{{\ensuremath{\boldsymbol{x}}}}

\newcommand{\matrixa}[0]{{\ensuremath{\mathbb{A}}}}
\newcommand{\Poly}[0]{{\ensuremath{\mathbb{P}}}}

\usepackage{units}
\newcommand{\fhalf}[0]{{\ensuremath{\nicefrac{1}{2}}}}

\newcommand{\nablah}[0]{{\ensuremath{\nabla_h}}}

\newcommand{\cdG}[0]{{\ensuremath{\textrm{cdG}}}}
\newcommand{\dG}[0]{{\ensuremath{\textrm{dG}}}}
\newcommand{\cG}[0]{{\ensuremath{\textrm{cG}}}}

\newcommand{\VcG}[1]{{\ensuremath{V_\cG^{#1}}}}
\newcommand{\VdG}[1]{{\ensuremath{V_\dG^{#1}}}}
\newcommand{\VcdG}[1]{{\ensuremath{V_\cdG^{#1}}}}
\newcommand{\Vorth}[1]{{\ensuremath{V_\perp^{#1}}}}
\newcommand{\cdGh}[0]{{\ensuremath{h}}}

\newcommand{\VtcdG}[1]{{\ensuremath{V_\cdG^{#1}}}}
\newcommand{\VtdG}[1]{\ensuremath{V_\dG^{#1}}}
\newcommand{\Vtorth}[1]{\ensuremath{V_\perp^{#1}}}
\newcommand{\Bform}[0]{{\ensuremath{\mathcal{B}}}}
\newcommand{\Sform}[0]{{\ensuremath{\mathcal{S}}}}

\newcommand{\pen}[0]{{\ensuremath{\mathrm{pen}}}}
\newcommand{\what}[0]{{\ensuremath{\hat{w}}}}
\newcommand{\vhat}[0]{{\ensuremath{\hat{v}}}}

\newcommand{\cont}[0]{\ensuremath{\cG}}
\newcommand{\discont}[0]{\ensuremath{\dG}}

\newcommand{\T}[0]{{\ensuremath{\mathcal{T}}}}
\newcommand{\Th}[0]{{\ensuremath{\mathcal{T}_h}}}

\newcommand{\TD}[0]{{\ensuremath{\mathcal{T}_{\discont}}}}
\newcommand{\TC}[0]{{\ensuremath{\mathcal{T}_{\cont}}}}

\newcommand{\TCj}[0]{{\ensuremath{\mathcal{T}^j_\cont}}}

\newcommand{\EinTh}[0]{{\ensuremath{E \in \Th}}}

\newcommand{\Eho}[0]{\ensuremath{{\mathcal{E}_h^o}}}

\newcommand{\Eh}[0]{\ensuremath{{\mathcal{E}_h}}}
\newcommand{\EhC}[0]{\ensuremath{{\mathcal{E}_\cont}}}
\newcommand{\EhD}[0]{\ensuremath{{\mathcal{E}_\discont}}}
\newcommand{\EhCj}[0]{\ensuremath{\mathcal{E}^j_\cont}}

\newcommand{\GC}[0]{\ensuremath{{\Gamma_\cont}}}
\newcommand{\GD}[0]{\ensuremath{{\Gamma_\discont}}}
\newcommand{\GCj}[0]{\ensuremath{{\Gamma_\cont^j}}}

\newcommand{\insymbol}[0]{\ensuremath{\text{in}}}
\newcommand{\outsymbol}[0]{\ensuremath{\text{out}}}
\newcommand{\Gout}[0]{\ensuremath{\Gamma^{\outsymbol}}}
\newcommand{\Gin}[0]{\ensuremath{\Gamma^{\insymbol}}}

\newcommand{\bout}[0]{\ensuremath{\partial^{\outsymbol}}}
\newcommand{\bin}[0]{\ensuremath{\partial^{\insymbol}}}
\newcommand{\einEh}[0]{{\ensuremath{e \in \Eh}}}
\newcommand{\einEho}[0]{{\ensuremath{e \in \Eho}}}
\newcommand{\einEhC}[0]{{\ensuremath{e \in \EhC}}}

\newcommand{\einEhCj}[0]{{\ensuremath{e \in \EhCj}}}

\newcommand{\pE}[0]{{\ensuremath{\partial E}}}
\newcommand{\pO}[0]{{\ensuremath{\partial \Omega}}}

\newcommand{\dx}[0]{{\ensuremath{\,\mathrm{d}\vecx}}}
\newcommand{\ds}[0]{{\ensuremath{\,\mathrm{d}s}}}

\newcommand{\DDnoarg}{{\ensuremath{\mathbb{D}}}}
\newcommand{\DD}[1]{{\ensuremath{\DDnoarg(#1)}}}
\newcommand{\KK}[0]{{\ensuremath{\mathbb{K}}}}
\newcommand{\EE}[1]{{\ensuremath{\mathbb{E}(#1)}}}

\newcommand{\ainv}[1]{{\ensuremath{a^{-1}(#1)}}}

\newcommand{\Euler}[0]{{\ensuremath{\mathfrak{d_t}}}}

\newcommand{\SH}[2]{{\ensuremath{H^{#1}(#2)}}}

\newcommand{\SL}[2]{{\ensuremath{L^{#1}(#2)}}}

\newcommand{\dealii}[0]{{\code{deal.ii}}}
\newcommand{\tol}[0]{{\ensuremath{\mathtt{tol}}}}

\begin{document}

\maketitle

%
\begin{abstract}
We prove in an abstract setting that standard (continuous) Galerkin finite element approximations are the limit of interior penalty discontinuous Galerkin approximations as the penalty parameter tends to infinity. We apply this result to equations of non-negative characteristic form and the non-linear, time dependent system of incompressible miscible displacement. Moreover, we investigate varying the penalty parameter on only a subset of a triangulation and the effects of local super-penalization on the stability of the method, resulting in a partly continuous, partly discontinuous method in the limit. An iterative automatic procedure is also proposed for the determination of the continuous region of the domain without loss of stability of the method.
\end{abstract}

%
\section{Introduction}  \label{sec:Introduction}

The discontinuous Galerkin (dG) finite element method has become widely used in recent years for a variety of problems as it possesses several desirable qualities, such as: Good stability properties due to the natural incorporation of upwinding techniques; flexible mesh design as hanging nodes and irregular meshes are admissible; and relatively easy implementation of $hp$-adaptive algorithms. These properties however come with the drawback of an increased number of degrees of freedom compared to a standard conforming method. For instance, when using an axi-parallel quadrilateral mesh in two dimensions with piecewise bilinear elements for which the standard continuous Galerkin (cG) finite element method has approximately $n$ degrees of freedom (depending on boundary conditions) the dG method on the same mesh has approximately $4n$ degrees of freedom.

For advection-dominated advection-diffusion-reaction equations the standard cG method exhibits poor stability properties and non physical oscillations may pollute the approximation globally. Discontinuous Galerkin methods have generally better stability properties. In the case of interior penalty dG method, for instance, stability in the upwind direction has been shown in the $\inf$-$\sup$ sense, e.g.,  in \cite{AM09,BHS06}, generalizing ideas from \cite{JP86}, where purely hyperbolic problems were considered. 

Conceptually, somewhere between the standard cG and interior penalty dG methods lies the continuous-discontinuous Galerkin (cdG) finite element method \cite{CGJ06}, whereby one seeks a Galerkin solution  on a finite element space $\VcdG{}$ with $\VcG{} \subset \VcdG{} \subset \VdG{}$, where $\VcG{}$ and  $\VdG{}$ are the standard cG and dG finite element spaces. In the context of problems with layers or sharp fronts, continuous elements can be used away from the layers/fronts and discontinuous elements (accommodating appropriate upwinding) can be used in the region where the layers/fronts are present. This idea has been studied previously in the context of problems with layers by Dawson and Proft \cite{DP02} using transmission conditions between regions where different spaces are used. Cangiani, Georgoulis and Jensen \cite{CGJ06} and Devloo, Forti and Gomez \cite{DFG07} have previously compared the cdG finite element method with alternative methods for advection-diffusion equations. 

The control of discontinuities across element interfaces in the dG framework can be exercised by introducing and/or tuning the, so-called, jump penalization parameters. Using excessive penalization within a dG approximation will be referred to as the \emph{super penalty method}. It is natural to expect that as the penalty parameter is increased the interelement jumps in the numerical approximation decrease. It has been shown by Larson and Niklasson \cite{LN00} for stationary linear elliptic problems (using the interior penalty method) and by Burman, Quarteroni and Stamm \cite{BQS10} for stationary hyperbolic problems (penalising the jumps of the approximation for discontinuous elements and the jumps in the gradient of the approximation for continuous elements) that the dG approximation converges to the cG approximation as the jump penalization parameter tends to infinity. 

In this work, our aim is twofold. Firstly, we present an alternative proof of the convergence of dG methods to cG methods, using a far more general framework covering the cases considered by \cite{BQS10,LN00} and also non-linear and time dependent problems. Moreover, we show that super-penalization procedures can be localized to designated element faces, thereby arriving to partly continuous, partly discontinuous finite element methods. As particular examples we consider the limits of the interior penalty dG method for PDEs with non-negative characteristic form \cite{HSS02} and the mixed Raviart-Thomas-dG method for the miscible displacement system presented in \cite{BJM09}.

Secondly, we continue the numerical investigations of \cite{CGJ06} in the context of blending locally continuous and discontinuous methods. In particular, we investigate to what extent numerical oscillations appear as local super-penalization is applied. The aim, of course, is to find the extent to which degrees of freedom can be removed by using locally continuous finite element spaces without affecting the extra stability offered by dG methods. To this end, we consider an advection-dominated advection-diffusion problem containing boundary layer behaviour, where the continuous and the discontinuous regions of the finite element solution are tuned manually. A second example investigates the use of an  iterative automatic procedure for the determination of the continuous region  of the domain by local super-penalization without loss of stability of the method. The procedure is applied to the problem of incompressible miscible displacement.

This work is organized as followed. After introducing notation in Section \ref{sec:Notation} an abstract discussion of the limit of penalty methods is given in Section \ref{sec:AbstractDiscussion}. We then show how this framework can be applied to equations of non-negative characteristic form in Section \ref{sec:LinearExample} and to the non-linear equations of incompressible miscible displacement in Section \ref{sec:NonLinearExample}. Finally, Section \ref{sec:Numerics} contains a number of numerical experiments and discussion of an iterative automatic procedure for determining the continuous regions of the approximation.

%
\subsection{Notation} \label{sec:Notation}

Let $\Omega\subset\Reals^d$ be a bounded, open polygonal domain. We 
denote by $\Th$ a subdivision of $\Omega$ into open non-overlapping $d$-simplices $E$. The diameter of $E\in\Th$ is denoted by $h_E$. Let also $\Eh:=\cup_{E\in\Th}\partial E$ be the skeleton of the mesh $\Th$, while $\Eho:=\Eh\backslash\partial\Omega$. Finally, let $\Gamma$ denote the set of elemental boundary faces, i.e., those which lie in $\pO$.

For $e \in \Eho$, with $e = \bar{E}^+ \cap \bar{E}^-$ for $E^+$, $E^- \in \Th$, we define $h_e := \min(h_{E^-},h_{E^+})$. Given a generic scalar field $\nu: \Omega \to \Reals$ that may be discontinuous across $e$, we set $\nu^\pm := \nu|_{E^\pm}$, the interior trace on $E^\pm$ and, similarly, for a generic vector field $\vectau : \Omega \to \Reals^d$. Define the average and jump for a generic scalar as
\begin{equation*}
	\Average{\nu} := \frac{1}{2} (\nu^+ + \nu^-), \qquad \Jump{\nu} := \nu^+ n^+ + \nu^- n^-, \qquad \mathrm{on}~ e \in \Eho, 
\end{equation*}
and for a generic vector field as
\begin{equation*}
	\Average{\vectau} := \frac{1}{2} (\vectau^+ + \vectau^-), \qquad \Jump{\vectau} := \vectau^+ \cdot n^+ + \vectau^- \cdot n^-, \qquad \mathrm{on}~ e \in \Eho,
\end{equation*}
where $n^\pm$ is the outward pointing normal from $E^\pm$ on $e$. For $e \in \Gamma$ the definitions become
\begin{equation*}
	\average{\nu} := \nu, \qquad \jump{\nu} := \nu n, \qquad \average{\vectau} := \vectau, \qquad \mathrm{on}~ e \in \Gamma. 
\end{equation*}

Given a vector $b$ denote the inflow and outflow boundaries of $\Omega$ by
\begin{align*}
		\bin \Omega \equiv \Gin := & \set{\vecx \in \pO : b \cdot n \le 0}, \\
		\bout \Omega \equiv \Gout := & \set{\vecx \in \pO : b \cdot n > 0} 
\end{align*}
and for an element
\begin{align*}
		\bin E := & \set{\vecx \in \pE : b \cdot n \le 0}, \\
		\bout E := & \set{\vecx \in \pE : b \cdot n > 0}.
\end{align*}
We denote the trace of a function $\nu$ on an edge by $\nu^\insymbol$ (resp.~$\nu^\outsymbol$) on the side of the edge where $b\cdot n \le 0$ (resp.~$b\cdot n > 0$). We construct the mesh so that the sign of $b \cdot n$ is the same for every $\vecx \in e$.

For the cdG method, we will require the following additional notation. We identify a decomposition of our triangulation $\Th$ into two disjoint triangulations $\TD$ and $\TC := \Th \setminus \TD$, upon which continuous and discontinuous elements will be applied respectively, henceforth referred to as the continuous and discontinuous regions of the triangulation. Define $J := \overline{\T}_\cG\cap\overline{\T}_\dG$ and define $\EhD$ to be the skeleton of $\TD$ and $\EhC := \Eh \setminus \EhD$. Note that with this definition the faces in $J$ are part of the discontinuous skeleton $\EhD$ only. Define $\GC := \TC \cap \Gamma$, the set of boundary faces of the continuous region, and similarly $\GD := \TD \cap \Gamma$.

Finally, we will denote by $\nablah$ the elementwise divergence operator.

%
\section{An abstract discussion} \label{sec:AbstractDiscussion}

Consider a (possibly non-linear) operator $\Bform:W\times W \rightarrow \Reals$ where $W$ is a finite dimensional vector space with norm $\norm{\,\cdotp}_W$. Suppose there exists a decomposition of $W$ such that $V \oplus X = W$ for $V, X \subset W$. In particular this means we can write any $w \in W$ uniquely as $w = v + x$ for some $v \in V$ and $x \in X$.

Assume that $\Bform$ is coercive, i.e., there exists $\Lambda_W > 0$ (typically independent of the dimension of $W$), such that
\begin{equation} \label{Bformcoercivity}
  \Bform(w,w) \ge \Lambda_W \norm{w}^2_W \qquad \forall~w \in W.
\end{equation}

Consider another operator $\Sform:W\times W \rightarrow \Reals$, whose support is restricted to $X \times X$ in the sense that
\begin{equation} \label{Sform1}
  \Sform(v, \vhat) = 0 \qquad \forall~ v,\vhat  \in V
\end{equation}
and
\begin{equation} \label{Sform2}
  \Sform(v,x) = \Sform(x, v) = 0 \qquad \forall~ v \in V, ~x \in X.
\end{equation}
We require coercivity on $X$, i.e., there exists $\Lambda_X >0$ such that for all $x \in X$
\begin{equation} \label{Sformcoercivity}
  \Sform(x, x) \ge \Lambda_X \norm{x}^2_X,
\end{equation}
where $\norm{x}_X$ is a norm on $X$. In view of \eqref{Sform1} this gives $\Sform(w,w) \ge \Lambda_X \norm{w}^2_X = \Lambda_X \norm{x}^2_X$. We construct a further operator
\begin{equation} \label{Bsigma}
  \Bform_\sigma := \Bform + \sigma \Sform
\end{equation}
where $0 \le \sigma \in \Reals$, and call this the \emph{super penalised} bilinear form.

Let $\ell$ be an element of the dual space $W^*$ of $W$, independent of $\sigma$. Then choose $w_\sigma \in W$ such that
\begin{equation} \label{wsigma}
  \Bform_\sigma (w_\sigma, w) = \ell(w) \qquad \forall~w\in W.
\end{equation}
Also choose $v_h \in V$ such that
\begin{equation} \label{vc}
  \Bform (v_h, v) = \ell (v) \qquad \forall~v \in V.
\end{equation}
Observe that for all $\sigma \in \Reals$
\begin{equation}
  \Bform_\sigma (v_h, v) = \Bform (v_h, v) = \ell (v) \qquad \forall~v \in V
\end{equation}
using \eqref{Sform1}. Now with \eqref{Bformcoercivity}, \eqref{Sformcoercivity} and \eqref{wsigma} we have
\begin{align*}
   \Lambda_W \norm{w_\sigma}^2_W + \sigma \Lambda_X \norm{w_\sigma}^2_X&\le \Bform(w_\sigma,w_\sigma) + \sigma \Sform(w_\sigma,w_\sigma)\\
  &= \Bform_\sigma (w_\sigma,w_\sigma)\\
  &= \ell(w_\sigma)\\
  &\le \norm{\ell}_{W^*}  \norm{w_\sigma}_W.
\end{align*}
Using Young's inequality we see
\begin{equation} \label{wsigmabounded}
  \frac{\Lambda_W^2}{\sigma} \norm{w_\sigma}^2_W + 2 \Lambda_W \Lambda_X \norm{w_\sigma}^2_X \le \frac{1}{\sigma} \norm{\ell}^2_{W^*}.
\end{equation}
Each of $\Lambda_W$, $\Lambda_X$ and $\norm{\ell}_{W^*}$ are independent of $\sigma$. We write $w_\sigma = v_\sigma + x_\sigma$, the unique decomposition with $v_\sigma \in V$ and $x_\sigma \in X$. From \eqref{wsigmabounded} we see
\begin{equation} \label{limitX}
  \lim_{\sigma \rightarrow \infty} \norm{v_\sigma + x_\sigma}_X = \lim_{\sigma \rightarrow \infty} \norm{x_\sigma}_X = 0.
\end{equation}
Therefore $x_\sigma \rightarrow 0$ as $\sigma \rightarrow \infty$.

Now assume that $\Bform$ is continuous in the first argument in the following sense: If $\lim_{i \rightarrow \infty} w_i = w \in W$ then
\begin{equation} \label{Bformcontinuity}
  \lim_{i \rightarrow \infty} \Bform(w_i,v) = \Bform(w,v) \qquad \forall~v\in V.
\end{equation}

Suppose $w_\sigma \nrightarrow v_h$ as $\sigma \to \infty$. Then there exists $\epsilon > 0$ such that there is some sequence $\set{w_{\sigma(i)}}_i$ with $\sigma(i) \to \infty$ as $i \to \infty$ satisfying
\begin{equation} \label{wsigma-contradiction}
  \norm{w_{\sigma(i)} - v_h}_W > \epsilon \qquad \forall i\in \Naturals.
\end{equation}
Owing to \eqref{wsigmabounded} the sequence $\set{w_{\sigma(i)}}_i$ is a bounded subset of $W$. Then by the Heine-Borel theorem there exists a convergent subsequence, also denoted $\set{w_{\sigma(i)}}_i$, such that
\begin{equation} \label{wtildelimit}
  \tilde{w} = \lim_{i \rightarrow \infty} w_{\sigma(i)}.
\end{equation}
Considering \eqref{limitX} we know that $\tilde{w} \in V$. We have that for all $v \in V$
\begin{align*}
  \Bform(\tilde{w},v) &= \Bform \left(\lim_{i \to \infty} w_{\sigma(i)}, v \right)\\
  &= \lim_{i \to \infty} \Bform(w_{\sigma(i)},v) &&\textrm{by }\eqref{Bformcontinuity}\\
  &= \lim_{i \to \infty} \Bform_\sigma (w_{\sigma(i)},v) &&\textrm{by }\eqref{Sform2}\\
  &= \lim_{i \to \infty} \ell(v) &&\textrm{by }\eqref{wsigma}\\
  &= \ell(v).
\end{align*}
Hence $\tilde{w}$ satisfies \eqref{vc} and by \eqref{wtildelimit} we have
\begin{align*}
  \lim_{i \to \infty} \norm{w_{\sigma(i)} - v_h}_W &= 0.
\end{align*}
This contradicts \eqref{wsigma-contradiction} and we conclude that all subsequences $\set{w_{\sigma(i)}}_i$ converge to $v_h$. Therefore 
\begin{equation}
  \lim_{\sigma \rightarrow \infty} (w_\sigma - v_h) = 0.
\end{equation}

We finally remark on the potential loss of stability due to super-penalization. It can be seen from \eqref{wsigmabounded} that as $x_\sigma\to 0$ when $\sigma\to \infty$ the coercivity of $\Bform_\sigma$ is increasingly compromised, which can lead to loss of stability and reduction on the rate of convergence in various settings.
%
\section{Equations of Non-Negative Characteristic Form} \label{sec:LinearExample}

We now examine the diffusion-advection-reaction equation (see \cite{HSS02})
\begin{equation} \label{ard-eqn}
  \begin{split}
  -\nabla \cdot (\matrixa(\vecx) \nabla u) + b(\vecx)\cdot \nabla u + c(\vecx)u &=f(\vecx) \qquad\text{in }\Omega,\\
  u&=0 \qquad\text{on }\pO
  \end{split}
\end{equation}
with $b$ a $\Reals^d$ valued function whose entries are Lipschitz continuous on $\overline{\Omega}$, $c \in \SL{\infty}{\Omega}$ and $f \in \SL{2}{\Omega}$ real valued functions. The diffusion coefficient $\matrixa$ is a $d\times d$ symmetric matrix with entries being bounded, piecewise continuous real-valued functions defined on $\overline{\Omega}$, with
\begin{align*}
  \zeta^\transpose \matrixa \zeta \ge 0 \qquad \forall \zeta \in \Reals^d, ~\mathrm{a.e.~}\vecx\in \overline{\Omega}.
\end{align*}
Under this condition, \eqref{ard-eqn} is named a \emph{partial differential equation with non-negative characteristic form}.

We define the cdG space to be
\begin{equation} \label{VcdG}
\begin{split}  
  \VcdG{} &:= \set{ v \in \SL{2}{\Omega} : \forall \EinTh,  v|_E \in \Poly^r, v|_{\GC} = 0, v|_{\TC} \in C(\overline{\T}_\cG)}
\end{split}
\end{equation}
where $\Poly^{r}$ is the space of polynomials of degree at most $r$ supported on $E$. We define the dG space to be
\begin{equation} \label{VdG}
  \VdG{} := \set{ v \in \SL{2}{\Omega} : \forall \EinTh, v|_E \in \Poly^r}.   
\end{equation}
Finally we define $\VdG{} := \VcdG{} \oplus \Vorth{}$ (corresponding to $W := V \oplus X$ in the notation of Section \ref{sec:AbstractDiscussion}). Note that the standard continuous space is obtained by setting $\Th = \TC$.

Define $\Bform:\VdG{} \times \VdG{} \to \Reals$, the bilinear form for the interior penalty family of methods with $\theta \in \set{-1,0,1}$ for \eqref{ard-eqn}, by
\begin{equation}
  \begin{split}
    \Bform(w,\what) &:= \Bform_d(w,\what) + \Bform_{ar}(w,\what)
  \end{split}
\end{equation}
with
\begin{equation}
  \begin{split}
    \Bform_d(w,\what) &:= \sum_\EinTh \int_E \matrixa \nablah w \cdot \nablah \what \dx + \sum_\einEh \int_e  m \jump{w}\cdot\jump{\what} \ds\\
    &\qquad - \sum_\einEh \int_e \Big(\average{\matrixa\nablah w}\cdot \jump{\what} - \theta \average{\matrixa\nablah \what}\cdot\jump{w} \Big) \ds 
  \end{split}
\end{equation}
and
\begin{equation}
  \begin{split}
    \Bform_{ar}(w,\what) &:= \sum_\EinTh \int_E (b \cdot \nablah w) \what + cw \what \dx\\
    &\qquad -\sum_\einEho \int_e b \cdot \jump{w} \what^\outsymbol \ds -  \sum_{e \in \Gin}\int_e (b\cdot n) w \what \ds.
  \end{split}
\end{equation}
We define $m := C_p \average{\overline{\matrixa} r^2}/h_e$, $\overline{\matrixa} := \norm{\abs{\sqrt{\matrixa}}_2}_\SL{\infty}{E}$, with $\abs{\,\cdotp}_2$ denoting the matrix-2-norm, and $C_p(\theta)\ge0$ fixed for a given $\theta$. The linear form is given by
\begin{equation}
  \ell(w) := \sum_\EinTh \int_E f w \dx.
\end{equation}
For $e \in \EhC$ we have the additional term $\Sform: \VdG{} \times \VdG{} \to \Reals$ penalising the jumps where
\begin{equation}
  \Sform(w,\what) := \sum_\einEhC \int_e M\jump{w}\cdot\jump{\what} \ds
\end{equation}
and
\begin{equation*}
  M := \left(C_{ar} + C_d \frac{\average{\overline{\matrixa} r^2}}{h_e} \right)
\end{equation*}
with $C_{ar}$ and $C_d$ fixed constants independent of $\sigma$. Then we define $\Bform_{\sigma}(w,\what) := \Bform(w,\what) + \sigma \Sform(w,\what)$.

Observe that if we take $C_{ar}= C_d= 0$ (or $\sigma = 0$) we recover the usual interior penalty method. If we take $C_p = C_d = C_{ar} = 0$ and $\matrixa = 0$ we have the standard (unpenalised) bilinear form for the purely hyperbolic equation (assuming of course that we adjust the boundary conditions appropriately). Taking $C_d = 0$ and $C_{ar} \ne 0$ when $\matrixa=0$ gives the method proposed in \cite{BMS04}.

All functions in $\VcdG{}$ are continuous on edges in $\EhC$ (recall that by definition edges in $J$ are not included in $\EhC$). Therefore conditions \eqref{Sform1} and \eqref{Sform2} are satisfied for this $\Sform$. That is, for any $v, \vhat \in \VcdG{}$ and $x \in \Vorth{}$
\begin{equation} \label{linearSproperties}
  \Sform(v, \vhat) = \Sform(v,x) = \Sform(x,v) = 0.
\end{equation}
We define the following norm for all $w \in \VdG{}$.
\begin{equation} \label{dGnormNNCF}
  \begin{split}
    \norm{w}_\dG^2 &:= \sum_\EinTh \norm{\sqrt{\matrixa}\nablah w}^2_\SL{2}{E} + \norm{c_0 w}^2_\SL{2}{\Omega} \\
    &\qquad +\sum_\einEh \frac{1}{2} \norm{\abs{b\cdot n}^\fhalf \jump{w}}^2_\SL{2}{e} + \sum_\einEh \norm{\sqrt{m} \jump{w}}^2_\SL{2}{e}
  \end{split}
\end{equation}
where $c_0 := \sqrt{c-\fhalf \nablah \cdot b}$. We also define for $w \in \VdG{}$
\begin{equation} \label{Snorm}
  \abs{w}^2_\Sform := \sum_\einEhC \norm{\sqrt{M}\jump{w}}^2_\SL{2}{e}.
\end{equation}
Notice that $\abs{\cdot}_\Sform$ is a semi-norm on $\VdG{}$ but a norm on $\Vorth{}$. To make this distinction clear we will write $\norm{x}_\Sform$ for $x \in \Vorth{}$.

\begin{lemma} \label{thm:linearBcoercivity}
  If $C_p$ is sufficiently large when $\theta = -1$ then $\Bform{}$ is coercive on $\VdG{}$, i.e., for all $w \in \VdG{}$
  \begin{equation} \label{linearBcoercivity}
    \Bform(w,w) \ge \Lambda_{cc} \norm{w}_\dG^2
  \end{equation}
  with $\Lambda_W = 1$ when $\theta = 1$ and $\Lambda_W = \fhalf$ when $\theta = -1$.
\end{lemma}

\begin{proof}
  See, e.g., \cite{HSS02} for a proof.
\end{proof}

From the definition it is clear that $\Sform$ is coercive with constant one on $\Vorth{}$, i.e., for all $x \in \Vorth{}$
\begin{equation} \label{linearScoercivity}
  \Sform(x, x) = \norm{x}^2_\Sform.
\end{equation}

\begin{definition}
  Define a dG approximation to \eqref{ard-eqn} as $w_\sigma \in \VdG{}$ satisfying
  \begin{equation} \label{lineardGFEM}
    \Bform_\sigma (w_\sigma,w) = \ell(w) \qquad \forall w \in \VdG{}.
  \end{equation}
\end{definition}
\begin{definition}
  Define a cdG approximation to \eqref{ard-eqn} as $v_h \in \VdG{}$ satisfying
  \begin{equation} \label{linearcdGFEM}
    \Bform_\sigma (v_h,v) = \ell(v) \qquad v \in \VcdG{}.
  \end{equation}
\end{definition}
Using \eqref{linearSproperties} we see that $v_h$ also satisfies $\Bform(v_h,v) = \ell(v)$ for all $v \in \VcdG{}$.

\begin{theorem} \label{thm:ARDconvergenceNNCF}
  The dG finite element approximation $w_\sigma$ converges to the cdG finite element approximation $v_h$ as $\sigma \to \infty$, i.e.,
  \begin{equation*}
    \lim_{\sigma \to \infty} (w_\sigma - v_h) = 0.
  \end{equation*}
\end{theorem}

\begin{proof}
Following the argument of Section \ref{sec:AbstractDiscussion} we use Lemma \ref{thm:linearBcoercivity} and \eqref{linearScoercivity} and note that \eqref{Bformcontinuity} is satisfied as linear operators in finite-dimensional vector spaces are continuous.
\end{proof}

\section[Incompressible Miscible Displacement]{Incompressible Miscible Displacement} \label{sec:NonLinearExample}

We consider the problem of finding the numerical solution to the coupled equations for the pressure $p=p(t,\vecx)$, Darcy velocity $u=u(t,\vecx)$ and concentration $c=c(t,\vecx)$ of one incompressible fluid in a porous medium being displaced by another. We consider the miscible case where both fluids are in the same phase.

Consider the domain $\Omega_T := (0,T) \times \Omega$. The equations for the miscible displacement are given by (e.g., \cite{Bas99,Bea88})
\begin{align}
	\phi \frac{\partial c}{\partial t} + u\cdot \nabla c - \nabla\cdot(\DD{u}\nabla c) + cq^I &= \hat{c}q^I \label{IMDtransport}, \\
	\nabla \cdot u &= q^I - q^P, \label{IMDcontinuity} \\
	u &= -\frac{\KK}{\mu (c)} \left( \nabla \cdot p - \rho (c)g \right) \label{IMDdarcy}
\end{align}
with the boundary conditions on $\pO_T := (0,T) \times \pO$ given by
\begin{align}
	u \cdot n &= 0 \label{IMDbc1}\\
	(\DD{u} \nabla c) \cdot n &=0, \label{IMDbc2}
\end{align}
and the initial conditions
\begin{equation}
	c(0,\cdot) = c_0 \label{IMDic}.
\end{equation}
We denote by: $\phi(\vecx)$ the porosity of the medium; $q^I\ge0$ and $q^P\ge0$ the pressure at injected (source) and production (sink) wells; $\KK(\vecx)$ the absolute permeability of the medium; $\mu(c)$ the viscosity of the fluid mixture; $\rho(c)$ the density of the fluid mixture; $g$ the constant vector of gravity; $\DD{u,\vecx}$ the diffusion-dispersion coefficient; $\hat{c}$ the injected concentration; and $c_0$ the initial concentration, which we assume for simplicity to be 0. We define $\ainv{c} := \KK^{-1} \mu$. The coupling is non-linear through the coefficients $\DD{u,\vecx}$, $\mu(c)$ and the advection term.  We make the common specific choice for the diffusion dispersion tensor, e.g., \cite{CE99,Fen95,SRW02}
\begin{equation} \label{DD-specific}
  \DD{u,\vecx} = \phi \left( d_m \Id + \abs{u}d_l \EE{u} + \abs{u}d_t (\Id - \EE{u})\right)
\end{equation}
where $\EE{u} = u u^\transpose / \abs{u}^2$ and $\Id$ is the identity matrix. We specify that the molecular, longitudinal and transverse diffusion coefficients $d_m$, $d_l$ and $d_t$ are positive real numbers.

We solve for the pressure and velocity using a Raviart-Thomas (RT) procedure \cite{Dur08,RT77} and for the concentration using a cdG method. We refer to the whole scheme as a RT-cdG method. For $k \ge 0$ we define
\begin{equation*}
  \begin{split}
  U &:= \set{v \in (\SL{2}{\Omega})^2: v|_E \in (\Poly_k(E))^2 + \vecx\Poly_k(E) ~\forall \EinTh,\\
  & \qquad v \cdot n \text{ continuous on }\einEho}.
  \end{split}
\end{equation*}
To avoid confusion for the pressure terms we define the space $P := \VdG{}$ where $\VdG{}$ is defined in \eqref{VdG}. Then the velocity and pressure are approximated in $U \times P$. To simplify the presentation we use the same mesh $\Th$ to solve for $u$, $p$ and $c$ numerically at each time step and there is no refinement of the mesh or polynomial degree. However $\TC$ and $\TD$ are not fixed so the cdG space used to approximate $c$ will vary with time. We define the time dependent cdG space by
\begin{equation*}
  \begin{split}
  \VtcdG{j} := \set{ v \in \SL{2}{\Omega} : \forall \EinTh,  v|_E \in \Poly^r, v|_{\GCj} = 0, v|_{\TCj} \in C(\overline{\T}_\cG^j)}
  \end{split}
\end{equation*}
where $\TCj$ and $\GCj$ are the $\TC$ region and external boundary of $\TCj$ at time $t_j$. As we assert that no change to the shape of the mesh occurs in time we define the time dependent dG space as in \eqref{VdG}. Then we may define $\VtdG{j} := \VtcdG{j} \oplus \Vtorth{j}$. Note that the degree $k$ is the same for $U$ and $P$ but need not be equal to $r$, the degree of the polynomials used to approximate concentration.

Let $0=t_0<t_1<\ldots< t_N =T$ be a partition of the time interval $(0,T)$. For simplicity we assume that each time step is of equal length and define $\Delta t:= t_j - t_{j-1}$ and the backward Euler operator $\Euler c_h^j := (\Delta t)^{-1} (c^j_h - c_h^{j-1})$ for $j=1,2,\ldots,N$. For the diffusion part of the concentration equation define the bilinear form
\begin{equation} \label{dGB_d}
	\begin{split}
		\Bform_d(c^j_h, d^j_h; u^j_h) &= \sum_{\EinTh} ( \DD{u^j_h} \nablah c^j_h, \nablah d^j_h )_E + \sum_\einEho m ( \jump{c^j_h},\jump{d^j_h} )_e  \\
		& \quad - \sum_\einEho \left[(\jump{c^j_h}, \average{\DD{u^j_h}\nablah d^j_h})_e + (\jump{d^j_h},\average{\DD{u^j_h}\nablah c^j_h})_e \right]
	\end{split}
\end{equation}
for all $d_h^j \in \VtcdG{j}$. The penalty parameter $m$ is defined by \cite{BJM09}
\begin{align*}
	m^2 : \Eh \to \Reals, \quad \vecx \mapsto C_\pen \frac{\max\{ n^\transpose_\Eh \DD{u_h^{j,+}, \vecx} n_\Eh , n^\transpose_\Eh \DD{u_h^{j,-}, \vecx} n_\Eh \}}{h}
\end{align*}
and $C_\pen$ is chosen such that it is larger than
\begin{align*}
	\sup \left\{ h \max \left\{ \frac{ \norm{\nu_h}^2_\pE}{\norm{\nu_h}^2_E} , \frac{\norm{D^\fhalf \nablah \nu_h}^2_\pE}{\norm{D^\fhalf \nablah \nu_h}^2_E} \right\} : \nu_h \in \Poly^s, D \in [\Poly^s]^{d \times d}, E ~\textrm{shape regular} \right\}.
\end{align*}

The bilinear form for convection, production and injection is given by the non-standard form
\begin{equation} \label{dGB_cq}
	\begin{split}
		&\Bform_{cq}(c^j_h,d^j_h;u^j_h) \\
		&= \frac{1}{2}\sum_\EinTh \left[ (u^j_h \cdot \nablah c^j_h, d^j_h)_E - (u^j_h c^j_h, \nablah d^j_h ) + ((q^I+q^P) c^j_h, d^j_h)_E \right] \\
		& \quad +\frac{1}{2}\sum_{\einEho} (u^j_h \cdot\jump{c^j_h}, d_h^{j,*})_e
	\end{split}
\end{equation}
where $d_h^{j,*}$ is defined by
\begin{equation}
	d_h^{j,*} = \left\{ \begin{array}{ll}
	                	d_h^{j,-} & \mathrm{if}~ u^j_h \cdot n^+ > 0,\\
	                	d_h^{j,+} & \mathrm{if}~ u^j_h \cdot n^+ \le 0.
	                \end{array} \right.
\end{equation}
This formulation ensures that $\Bform_{cq}$ is semi-definite regardless of the properties of $u_h^j$. We do not need to restrict sums over edges to cells in $\TC$ as in this region elements of $\VcdG{}$ are continuous. Also note that the dG method is a special case of the cdG method where $\TC = \emptyset$.

Define $\Bform(c^j_h,d^j_h;u^j_h) := \Bform_d(c^j_h, d^j_h; u^j_h) + \Bform_{cq}(c^j_h,d^j_h;u^j_h)$ and
\begin{equation}
  \Sform(c^j_h,d^j_h) := \sum_\einEhCj \int_e M \jump{c^j_h}\cdot\jump{d^j_h} \ds
\end{equation}
where
\begin{equation*}
  M := \left(C_d \frac{r^2}{h_e} \right).
\end{equation*}
Then for any $c_h^j, d_h^j \in \VtcdG{j}$ and $x^j \in \Vtorth{j}$
\begin{equation*}
  \Sform(c_h^j, d_h^j) = \Sform(c_h^j, x^j) = \Sform(x^j, c_h^j) = 0.
\end{equation*}
We define the following norm for $u_h^j \in U$ and $c_h^j \in \VtdG{j}$:
\begin{equation}
  \begin{split}
    \triple{c^j_h}^2 &:= \sum_\EinTh \norm{\sqrt{\DD{u^j_h}}\nablah c^j_h}^2_\SL{2}{E} + \frac{1}{2}\norm{q_0 c^j_h}^2_\SL{2}{\Omega} \\
    &\qquad +\sum_\einEho \frac{1}{2} \norm{\abs{u^j_h\cdot n}^\fhalf \jump{c^j_h}}^2_\SL{2}{e} + \sum_\einEho \norm{\sqrt{m} \jump{c^j_h}}^2_\SL{2}{e}
  \end{split}
\end{equation}
where $q_0 := \sqrt{q^I + q^P}$. For $c_h^j \in \VtdG{j}$ define
\begin{equation} \label{SnormNL}
  \abs{c^j_h}^2_\Sform := \sum_\einEhCj \norm{\sqrt{M}\jump{c^j_h}}_\SL{2}{e}.
\end{equation}
Notice that \eqref{SnormNL} is a semi-norm on $\VtdG{j}$ but a norm on $\Vtorth{j}$.

\begin{lemma} \label{thm:nonlinearBcoercivity}
  If $C_\pen$ is chosen large enough then $\Bform$ is coercive for all $c_h^j \in \VtdG{j}$ and $u_h^j \in U$, i.e.,
  \begin{equation} \label{nonlinearBcoercivity}
    \Bform(c^j_h,c^j_h;u^j_h) \ge \Lambda_W \triple{c_h^j}^2.
  \end{equation}
\end{lemma}

\begin{proof}
  Combine equations (4.3) and (4.6) from \cite{BJM09}.
\end{proof}

We have by construction that $\Sform$ is coercive with constant one on $\Vtorth{j}$, i.e., for all $x^j_h \in \Vtorth{j}$
\begin{equation} \label{nonlinearScoercivity}
  \Sform(x^j_h, x^j_h) = \norm{x^j_h}^2_\Sform.
\end{equation}

We discretise the time derivative with the backward Euler operator. Summing over each discrete time step gives
\begin{align*}
  \sum_{j=1}^N (\phi \Euler c_h^j,c_h^j) &= \sum_{j=1}^N \frac{1}{\Delta t} (\phi c_h^j,c_h^j) - \frac{1}{\Delta t} (\phi c_h^{j-1},c_h^j)\\
  &\ge \sum_{j=1}^N \frac{1}{\Delta t} \norm{\phi^\fhalf c_h^j}^2_\SL{2}{\Omega} - \frac{1}{2\Delta t}\left( \norm{\phi^\fhalf c_h^{j-1}}^2_\SL{2}{\Omega} + \norm{\phi^\fhalf c_h^{j}}^2_\SL{2}{\Omega}\right)\\
  &=\frac{1}{2\Delta t} \left( \norm{\phi^\fhalf c_h^N}^2_\SL{2}{\Omega} - \norm{\phi^\fhalf c_h^0}^2_\SL{2}{\Omega} \right)
\end{align*}
where we have used Young's Inequality. We have assumed that the initial concentration is 0 and so $\norm{\phi^\fhalf c_h^0}^2_\SL{2}{\Omega}=0$.

\begin{definition}\label{def:dGFEMAlgorithm}
  Define the RT-dG approximation $(u_h,p_h,c_{\sigma}) \in \Pi_{j=1}^N U \times \Pi_{j=1}^N P \times \Pi_{j=1}^N \VtdG{j}$ to \eqref{IMDtransport}-\eqref{IMDic} as that generated by the algorithm: For $1\le j \le N$ and $c_{\sigma}^{j-1} \in \VtdG{j}$ find $(u_h^j,p_h^j,c_{\sigma}^j) \in  U \times P \times \VtdG{j}$ such that
  \begin{align}
	  (\nablah \cdot u^j_h,w^j_h) &= (q^I - q^P, w^j_h), \label{dGFEcontinuity} \\
	  (\ainv{c^j_{\sigma}} u^j_h, v^j_h) - (p^j_h, \nablah \cdot v^j_h) &= (\rho(c^j_{\sigma})g, v^j_h) \label{dGFEdarcy} 
  \end{align}
  for all $(v^j_h,w^j_h) \in U \times P$ and
\begin{equation}\label{dGFEtransport}
\begin{split}
	&\left(\phi \Euler c_{\sigma}^j, d^j_h\right) + \Bform(c^j_{\sigma}, d^j_h; u^j_h)+ \sigma \Sform(c^j_{\sigma}, d^j_h) = (\hat{c} q^{I}, d^j_h) 
\end{split}
\end{equation}
for all $d^j_h \in \VtdG{j}$. 
\end{definition}

\begin{definition} \label{def:cdGFEMAlgorithm}
  Define the RT-cdG approximation $(u_h,p_h,c_\cdGh) \in \Pi_{j=1}^N U \times \Pi_{j=1}^N P \times \Pi_{j=1}^N \VtcdG{j}$ to \eqref{IMDtransport}-\eqref{IMDic} as that generated by the algorithm: For $1\le j \le N$ and $c_\cdGh^{j-1} \in \VtcdG{j}$ find $(u_h^j,p_h^j,c_\cdGh^j) \in  U \times P \times \VtcdG{j}$ such that
  \begin{align}
	  (\nablah \cdot u^j_h,w^j_h) &= (q^I - q^P, w^j_h), \label{cdGFEcontinuity} \\
	  (\ainv{c^j_\cdGh} u^j_h, v^j_h) - (p^j_h, \nablah \cdot v^j_h) &= (\rho(c^j_\cdGh)g, v^j_h) \label{cdGFEdarcy} 
  \end{align}
  for all $(v^j_h,w^j_h) \in U \times P$ and
\begin{equation}
  \begin{split}
	&\left(\phi \Euler c_\cdGh^j, d^j_h\right) + \Bform(c^j_\cdGh, d^j_h; u^j_h) + \sigma \Sform(c^j_\cdGh, d^j_h) = (\hat{c} q^{I}, d^j_h) \label{cdGFEtransport}
  \end{split}
\end{equation}
for all $d^j_h \in \VtcdG{j}$.
\end{definition}

\begin{theorem} \label{thm:convergenceIMD}
  The solution $c_\sigma \in \Pi_{j=1}^N \VtdG{j}$ defined in Definition \ref{def:dGFEMAlgorithm} converges to $c_\cdGh \in \Pi_{j=1}^N \VtcdG{j}$ defined in Definition \ref{def:cdGFEMAlgorithm} as $\sigma \to \infty$, i.e.,
  \begin{equation}
    \lim_{\sigma \to \infty} (c_\sigma - c_h) = 0.
  \end{equation}
\end{theorem}

\begin{proof}
Following the argument of Section \ref{sec:AbstractDiscussion} we use Lemma \ref{thm:nonlinearBcoercivity} and \eqref{nonlinearScoercivity}. In order to complete the proof using this argument we must show that for every sequence $\set{c_i^j}_i$ with elements in $\VtdG{j}$ and $\lim_{i\to\infty} c^j_i = c^j \in \VtdG{j}$ we have
\begin{equation} \label{IMDBformcontinuity}
  \lim_{i\to\infty} \Bform(c_i^j, d^j_h; u^j (c_i^j)) = \Bform(c^j, d^j_h; u^j(c^j)) \qquad \forall d_h^j \in \VtdG{j}
\end{equation}
as in \eqref{Bformcontinuity}, where $u^j(\,\cdotp)$ is the element in $U$ solving \eqref{dGFEcontinuity}-\eqref{dGFEdarcy} for a given element of $\VtdG{j}$.  Note that $u^j : \VtdG{j} \to U$ is a continuous map and so $\lim_{i\to\infty} u^j (c^j_i) =  u^j (\lim_{i\to\infty} c^j_i) = u^j(c^j)$. This also holds for derivatives as they are taken piecewise. Therefore \eqref{IMDBformcontinuity} holds at each timestep and for the whole discrete solution in time.
\end{proof}

%
\section{Numerical Experiments} \label{sec:Numerics}

We present numerical experiments to illustrate Theorems \ref{thm:ARDconvergenceNNCF}, \ref{thm:convergenceIMD} and investigate further the performance of the cdG method.

The results were produced using the \CC library {\dealii}  \cite{BHK07,DealIIRef} using both the super penalty approach (as $\sigma \to \infty$) and a \emph{direct} cdG method, i.e., where the test functions are in \VcdG{} and therefore by construction there will be no jumps across edges in $\EhC$. For further details of the implementation we refer to \cite{CCGJ12Enumath}.

%
\subsection{Equations of Non-negative Characteristic Form} \label{sec:LinearEx0423a}

Let $\Omega = (0,1)^2$. We seek to solve
\begin{equation*}
  -\epsilon \Delta u + (1,1)\cdot \nabla u  = f.
\end{equation*}
Given homogeneous Dirichlet boundary conditions $f$ is chosen such that the solution is given by
\begin{equation*}
  u(x,y) := \left(x - \frac{\ex^{(x-1)/\epsilon} - \ex^{-1/\epsilon}}{1-\ex^{-1/\epsilon}}\right)\left( y- \frac{\ex^{(y-1)/\epsilon} - \ex^{-1/\epsilon}}{1-\ex^{-1/\epsilon}}\right).
\end{equation*}
For $0<\epsilon \ll 1$ this problem exhibits exponential boundary layers along the outflow boundaries $x=1$ and $y=1$ of width $\Order{\epsilon}$. We consider a uniformly refined mesh of squares and set $r=1$ (piecewise bilinear polynomials).

We first look at an example without a layer by setting $\epsilon = 10$. We set $\TC = \Th$, i.e., the cG method. Figure \ref{fig:0423aEx6} shows the behaviour of the difference between the dG and cG approximations in the $\SL{2}{\Th}$ norm, $\SH{1}{\Th}$ semi-norm and the $L^2$ norm of the jumps across edges (represented by $\jump{\,\cdotp}$). As $\sigma$ grows the difference in each norm decreases linearly. The jumps in the either approximation are already very small, i.e., the dG approximation is very close to an element in the cG space. We do not see oscillations polluting the continuous approximation.

\begin{figure}[htb]
  \centering
  \psfrag{sigma}[rB]{$\sigma$}
  \psfrag{penalty}[rB]{$\sigma$}
  \psfrag{diffnorm}[cb][c]{$\norm{w_\sigma - v_h}$}
  \psfrag{truenorm}[cb][c]{$\norm{u-w_\sigma}$}
  \psfrag{L2}[c][r]{\tiny{$L^2$}}
  \psfrag{H1}[c][r]{\tiny$H^1$}
  \psfrag{L2jump}[r][r]{\tiny{$\jump{\,\cdotp}$}\hspace{20pt}}
  \includegraphics[width=0.70\textwidth]{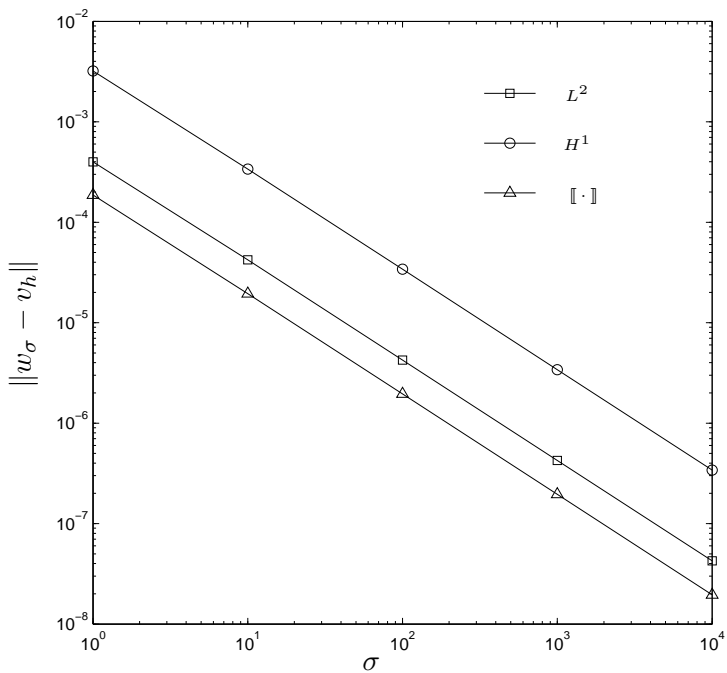}
  \caption{Example \ref{sec:LinearEx0423a} with $\epsilon = 10$ and $\TC = \Th$. As the penalty parameter is increased the difference between the cG and dG approximations decreases linearly in the given norms.}
  \label{fig:0423aEx6}
\end{figure}

We now motivate the cdG method by choosing $\epsilon = 10^{-4}$ and again setting $\TC = \Th$.  The example now has a sharp layer at the outflow boundaries. We see in Figure \ref{fig:0423aEx7}\subref{fig:0423aEx7diff} that increasing $\sigma$ gives a linear response to the error as in Figure \ref{fig:0423aEx6}. When we look at the error in the dG approximation in Figure \ref{fig:0423aEx7}\subref{fig:0423aEx7true} we see that the approximation becomes worse as the penalty is increased. The layer causes non-physical oscillations to pollute the approximation. Although we see convergence of the dG approximation to the cG approximation this property is not desirable.

\begin{figure}[htb]
  \centering
  \psfrag{sigma}[rB]{$\sigma$}
  \psfrag{penalty}[rB]{$\sigma$}
  \psfrag{diffnorm}[cb][c]{$\norm{w_\sigma - v_h}$}
  \psfrag{truenorm}[cb][c]{$\norm{u-w_\sigma}$}
  \psfrag{L2}[c][r]{\tiny{$L^2$}}
  \psfrag{H1}[c][r]{\tiny$H^1$}
  \psfrag{L2jump}[r][r]{\tiny{$\jump{\,\cdotp}$}\hspace{8pt}}
  \subfigure[The difference between the cG and dG approximations.]{\includegraphics[width=0.45\textwidth]{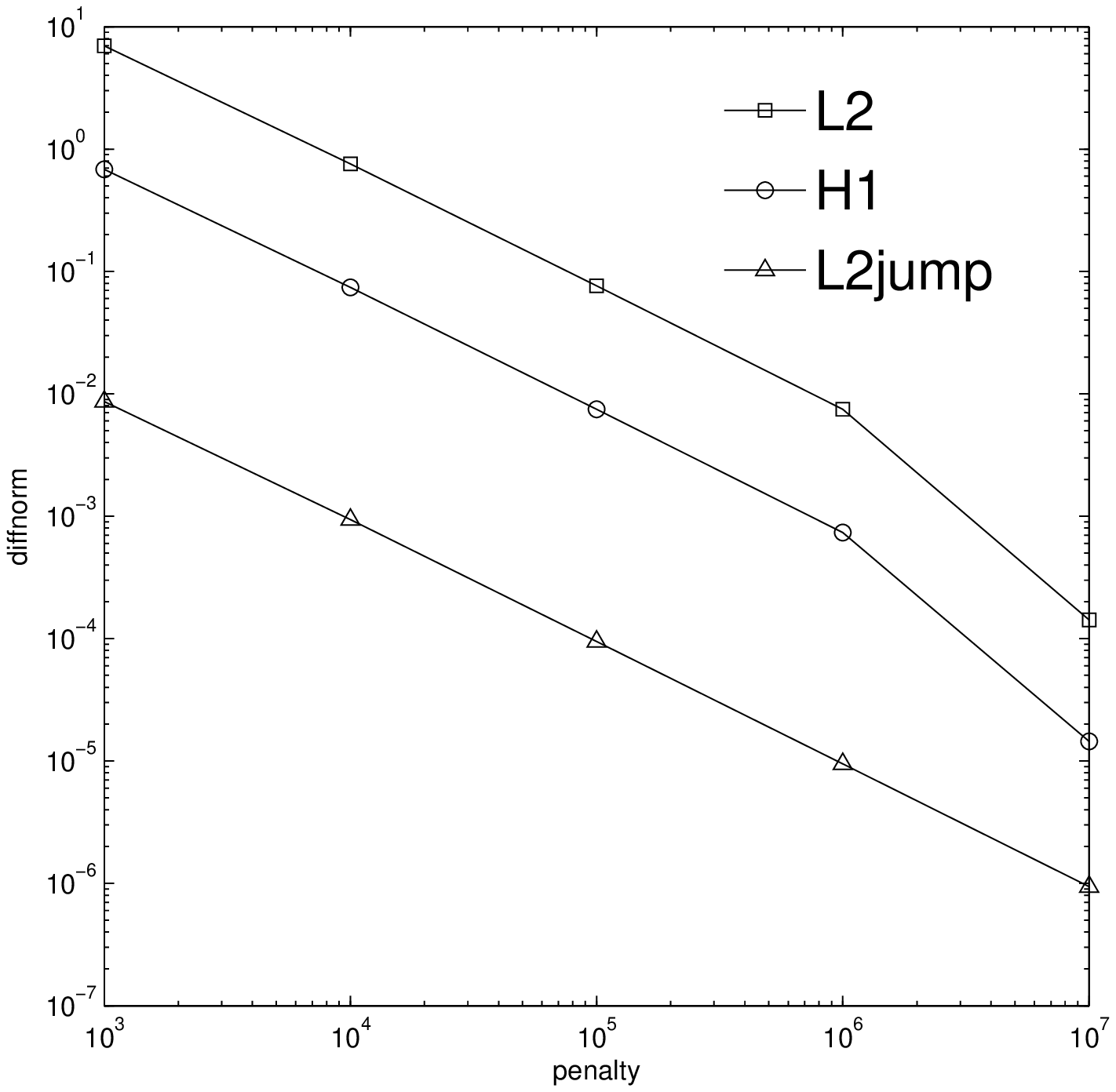} \label{fig:0423aEx7diff} }\hspace{10pt}
  \subfigure[The error in the dG approximation.]{\includegraphics[width=0.45\textwidth]{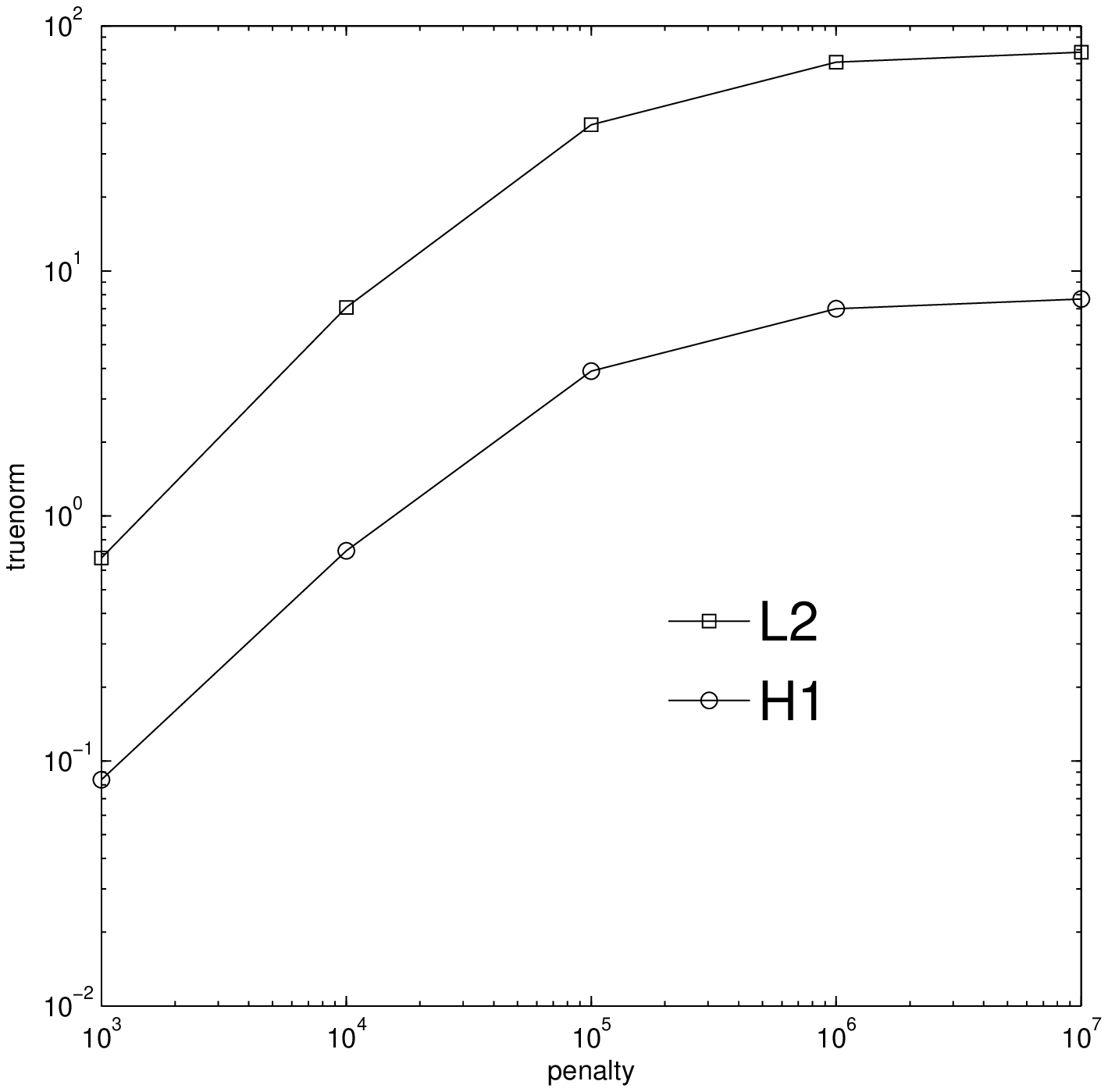} \label{fig:0423aEx7true}}
  \caption{Example \ref{sec:LinearEx0423a} with $\epsilon = 10^{-4}$ and $\TC = \Th$. Now the problem has a layer the error in the dG approximation grows as $\sigma$ is increased. Non-physical oscillations pollute the approximation.}
  \label{fig:0423aEx7}
\end{figure}

We now consider the cdG method with $\epsilon  = 5\times 10^{-3}$ and $5\times10^{-4}$, values chosen so that the layer is partially resolved in the first case and not resolved in the second. The behaviour as $\sigma$ is increased is the same as in the case $\Th = \TC$ and so we do not plot this. We set $h=2^{-5}$ and $\TC = (0,1-ah)^2$. Varying $a \in \Integers$ determines the number of rows in the dG region at the outflow boundary. For $\epsilon = 5\times 10^{-3}$ oscillations are apparent in the cG approximation but the mesh is sufficiently refined so that they are not large. Decreasing $a$ (that is, moving from a fully discontinuous approximation towards a fully continuous approximation) results in a small increase in the error of the approximation which can be seen in Figure \ref{fig:0423acdG}. The continuous, discontinuous and continuous-discontinuous approximations are very close in the $H^1$ semi-norm, including when $\TC$ covers the layer. When $\epsilon = 5\times 10^{-4}$ the layer is sufficiently sharp to induce large oscillations in the fully continuous approximation. By choosing the $\TC$ region to allow discontinuities in at least the final element we may achieve a reduction in degrees of freedom to approximately 30\% of the discontinuous approximation with very little effect on the error ($4096$ degrees of freedom for the dG approximation, $1024$ for the cdG approximation and $1276$ for the cdG approximation with one row of dG elements at the outflow boundary).

\begin{figure}[htb]
  \centering
  \psfrag{aa}[rB]{$a$}
  \psfrag{H1norm}[cb][c]{$\norm{u-v_h}_\SH{1}{\Omega}$}
  \psfrag{eps3}[b][r]{\tiny{$\epsilon = 5\times10^{-3}$}}
  \psfrag{eps4}[b][r]{\tiny{$\epsilon = 5\times10^{-4}$}}
  \includegraphics[width=0.70\textwidth]{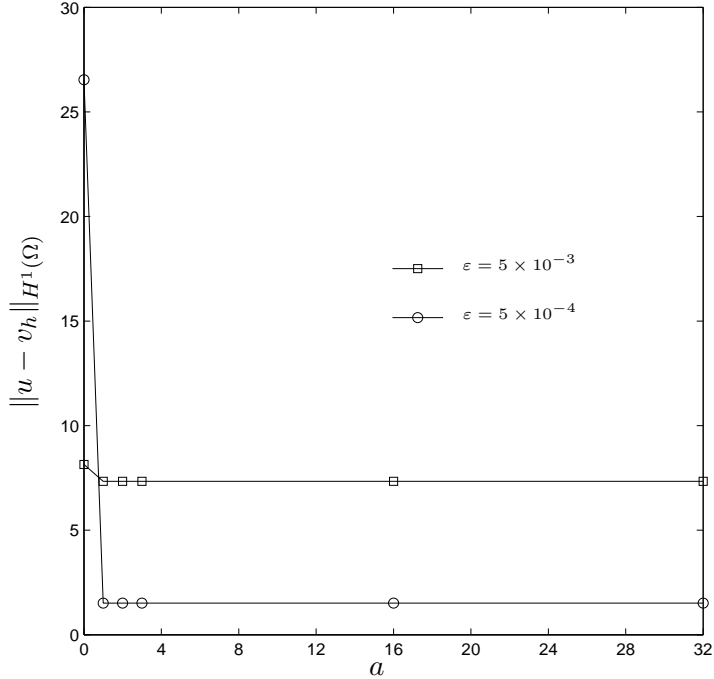}
  \caption{Example \ref{sec:LinearEx0423a} varying $\TC$ for a partially resolved and unresolved layer. In both cases the increase in the error as $\TC$ covers the layer ($a=0$) is apparent. When the layer is sharper the increase is more pronounced as severe oscillations pollute the approximation.}
  \label{fig:0423acdG}
\end{figure}


%
\subsection{Incompressible Miscible Displacement: The ``Quarter of Five Spot" Problem} \label{sec:NonlinearNumerics}

As well as verifying Theorem \ref{thm:convergenceIMD} we wish to show that if the region where continuous elements are used is chosen appropriately there is little difference in the approximations via the RT-cdG or RT-dG method (where the concentration is approximated in the dG space).

We study a standard example \cite{BJM09,CJ86,SRW02} to illustrate the performance of the cdG method for the incompressible miscible displacement problem \eqref{IMDtransport}-\eqref{IMDic}. With $\Omega = (0,1)^2$ the injection (resp.~extraction) well is located at $(1,1)$ (resp.~$(0,0)$). The injection and extraction strength are represented over one element by piecewise constant functions such that $\int_\Omega q^I \dx = \int_\Omega q^P \dx = 0.018$. In \eqref{DD-specific} we set $d_l = 1.8\times 10^{-4}$, $d_m = 1.8\times 10^{-6}$ and $d_t = 1.8\times 10^{-5}$. The porosity is set to $0.1$. The concentration dependent viscosity is given by $\mu(c) = \mu(0)(1+(\mathcal{M}^{1/4} - 1)c)^{-4}$ where $\mathcal{M} = 41.0$ is the mobility ratio (the ratio of the viscosity of the fluids), and $\mu(0) = 1$. For the initial concentration we set $c_0 = 0$ corresponding to $\Omega$ uniformly filled with one fluid. Set $\KK = 0.0288\Id$. We consider a uniform refinement of $\Omega$ into squares of side $h = 2^{-4}$ with timestep $4\times 10^{-3}$ and time interval $(0.0,2.0)$. With these values a sharp front in the concentration component spreads from the injection to extraction point. As can be seen in Figure \ref{fig:0430b}\subref{fig:0430bcG} this causes oscillations in the continuous approximation.

First we present the difference between the dG approximation and the cG approximation (i.e., with $\TC = \Th$) as $\sigma \to \infty$. In Figure \ref{fig:0501a} we show $\norm{c_\sigma - c_h}$ in both the $L^2$ norm against time and the $\SL{2}{(0,T);\SL{2}{\Omega}}$ norm against increasing $\sigma$. In Figure \ref{fig:0501a}\subref{fig:0501atimeerror} we see a sharp increase in the error over the first few iterations. The initial conditions are in the continuous approximation space so the cG and dG approximations are close. As the layer spreads through the domain the difference between the cG and dG approximations for a given $\sigma$ in the $L^2$ norm increases slowly. This is because the number of edges in the vicinity of the layer increases. Figure \ref{fig:0501a}\subref{fig:0501asigmaerror} shows the same behaviour as the stationary examples in Section \ref{sec:LinearEx0423a}.

\def\supertiny{\fontsize{4pt}{10pt}\selectfont}

\begin{figure}[htb]
  \centering
  \psfrag{time}[rB]{$t$}
  \psfrag{sigma}[B]{$\sigma$}
  \psfrag{diffnorm}[cb][c]{\tiny$\norm{c^j_\sigma - c^j_h}_\SL{2}{\Omega}$}
  \psfrag{errornorm}[cb][c]{\tiny$\norm{c_\sigma - c_h}_\SL{2}{(0,T);\SL{2}{\Omega}}$}
  \psfrag{1e3}[c][r]{{\supertiny$10^{3}$}\hspace{10pt}}
  \psfrag{1e5}[c][r]{\supertiny{$10^{5}$}\hspace{10pt}}
  \psfrag{1e7}[c][r]{\supertiny{$10^{7}$}\hspace{10pt}}
  \psfrag{1e9}[c][r]{\supertiny{$10^{9}$}\hspace{10pt}}
  \psfrag{1e11}[c][r]{\supertiny{$10^{11}$}\hspace{10pt}}
  \subfigure[Evolution of the difference between the cG and dG approximations for $\sigma = 10^3$ to $10^{11}$ in $L^2$ norm.]{\includegraphics[width=0.45\textwidth]{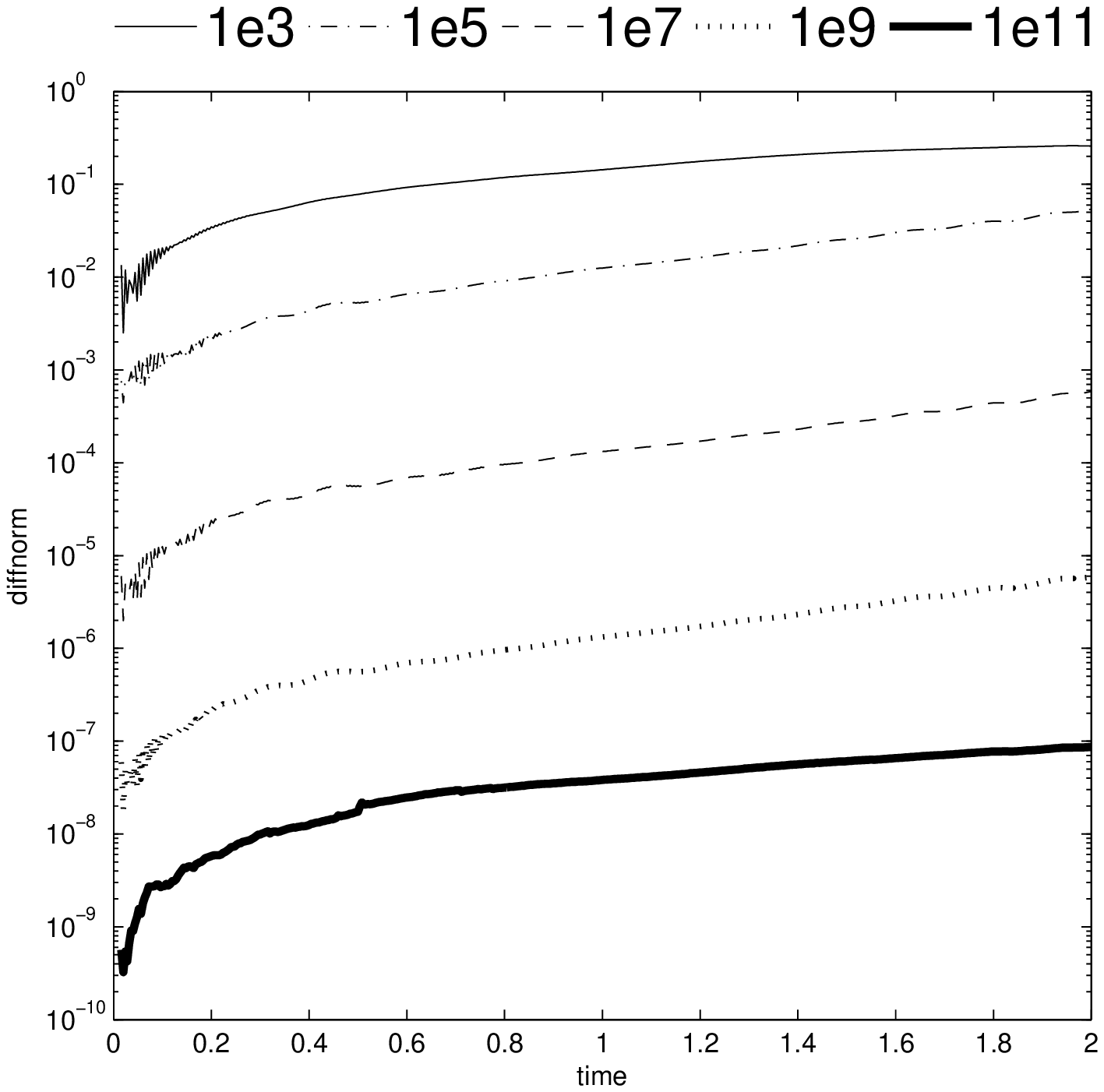} \label{fig:0501atimeerror} }\hspace{10pt}
  \subfigure[Plot of the difference between the cG and dG approximations in the $L^2(L^2)$ norm as $\sigma$ is increased.]{\includegraphics[width=0.45\textwidth]{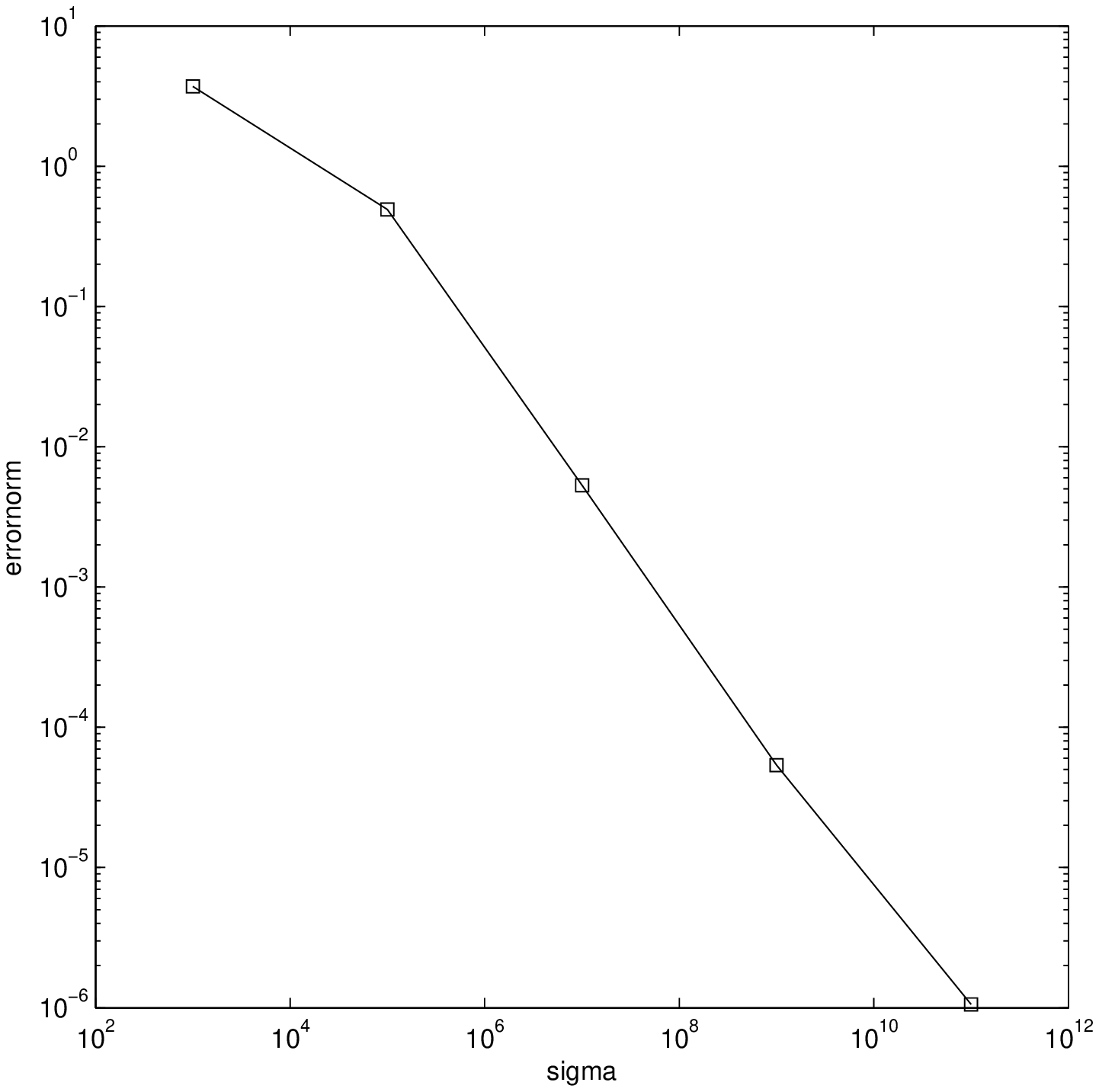} \label{fig:0501asigmaerror}}
  \caption{The effect of increasing $\sigma$ for Example \ref{sec:NonlinearNumerics} with $\TC = \Th$.}
  \label{fig:0501a}
\end{figure}

Picking $\TC$ in Section \ref{sec:LinearEx0423a} was done via knowledge of the true solution and hence knowledge of any layers. We do not have this luxury for the problem considered in this section. We therefore undertake the following procedure for determining $\TC$:
\begin{enumerate}[label=(\arabic*),ref=(\arabic*)]
  \item Determine the initial pressure and velocity given $c_h^0$ and the injection profile.
  \item Solve for the first time step using a RT-dG method to find a discontinuous $c_h^1$.
  \item For all edges determine $\norm{\jump{c_h}}_\SL{2}{e}$.\label{refinestep}
  \item Flag every cell where each edge satisfys $\norm{\jump{c_h}}_\SL{2}{e} < \tol$.
  \item If every edge of a cell is flagged set that cell to be part of $\TC$ in the next iteration. Otherwise the element will be in $\TD$.
  \item For $n$ iterations use the cdG mesh defined in the previous step.
  \item For the $(n+1)^{th}$ iteration reset the mesh to be entirely dG, i.e., $\T^{n+1}_\cG = \emptyset$ for the concentration component, then return to step \ref{refinestep}.
\end{enumerate}

The number of iterations between each cdG refinement and the tolerance should consider the expected motion of the fluid and the time step. We do not consider increasing $\sigma$ for the cdG method, but rather study the performance of the method as the tolerance is increased by comparing the cdG approximation with a dG approximation where $C_\pen = 10$ and $\sigma=0$. With these parameters we set the number of iterations between redefining the cdG space to be 5.

In Figure \ref{fig:0430a} we see that as the tolerance is decreased the difference between the dG and cdG approximations in the $L^2$ norm gets smaller. With a smaller tolerance fewer cells are marked as being continuous. The difference introduced by using some continuous elements does not seem to propagate in time.

\begin{figure}[htb]
  \centering
  \psfrag{time}[rB]{t}
  \psfrag{L2diffnorm}[cb][c]{\tiny$\norm{c^j_\sigma - c^j_h}_\SL{2}{\Omega}$}
  \psfrag{cG}[c][r]{\tiny{cG}}
  \psfrag{cdG-3}[c][r]{\tiny{$\tol = 10^{-3}$}\hspace{10pt}}
  \psfrag{cdG-4}[c][r]{\tiny{$\tol = 10^{-4}$}\hspace{10pt}}
  \psfrag{cdG-5}[c][r]{\tiny{$\tol = 10^{-5}$}\hspace{10pt}}
  \includegraphics[width=0.70\textwidth]{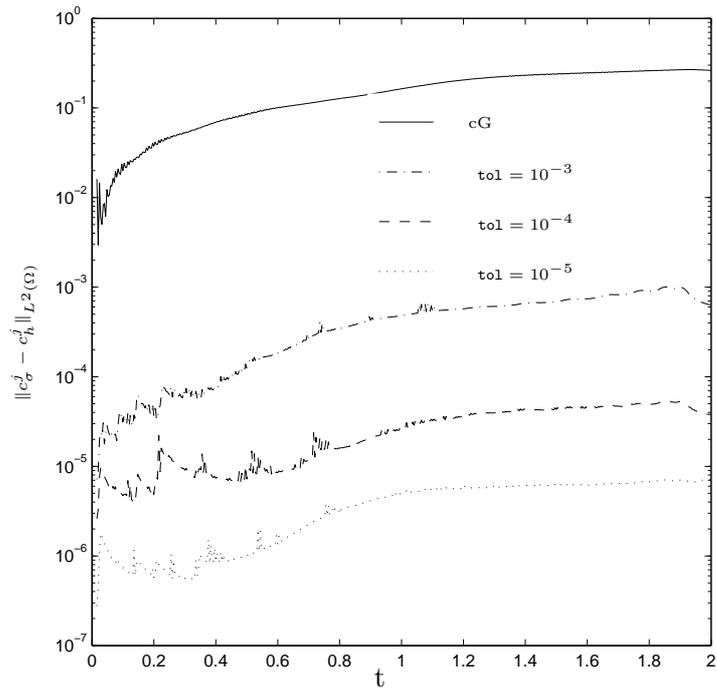}
  \caption{The behaviour of the cdG approximation for Example \ref{sec:NonlinearNumerics}. Using some continuous elements does not dramatically increase the error of the cdG approximation compared to the dG approximation.}
  \label{fig:0430a}
\end{figure}

In Table \ref{table:0430} we see that the number of degrees of freedom saved over the simulation (500 steps with $T= 2.0$, $\Delta t = 4\times 10^{-3}$) is considerable. The effect on the approximation is however small measured in the $L^2(L^2)$ norm. The number of degrees of freedom for the cG method is not 128,000 as would be expected (one degree of freedom per vertex on a $16\times 16$ square mesh for 500 timesteps) due to every fifth iteration being discontinuous.

\begin{table}[htb] 
\centering
\begin{tabular}{|c|c|c|}
\hline
$\tol$ & dofs & $\norm{c_\sigma - c_h}_\SL{2}{(0,T);\SL{2}{\Omega}}$\\
\hline\hline
cG & 219,470 & $3.9970\times 10^0$ \\
\hline 
$10^{-3}$ & 323,488 & $1.2073\times 10^{-2}$ \\
\hline
$10^{-4}$ & 355,328 &  $7.0904\times 10^{-4}$\\
\hline
$10^{-5}$ & 382,384 & $1.0455\times 10^{-4}$ \\
\hline
dG & 512,000 & $0.0000\times 10^0$ \\
\hline
\end{tabular}\vspace{10pt}
\caption{The number of degrees of freedom used for 500 timesteps in Example \ref{sec:NonlinearNumerics}. When $\tol = 10^{-5}$ only 74.6\% of the degrees of freedom are used compared to 43\% for the continuous approximation.}
\label{table:0430}
\end{table}

In Figure \ref{fig:0430b} we show the dG, cG and cdG approximations after 380 timesteps. There is no visible difference between the plots for dG and cdG at each tolerance (Figures \ref{fig:0430b} \subref{fig:0430bdG}, \subref{fig:0430bcdG} and \subref{fig:0430bcdG2}). However for the fully continuous approximation the oscillations induced by the layer are clearly visible and distort the plot.

\begin{figure}[htb]
  \centering
  \subfigure[The fully discontinuous approximation.]{\includegraphics[width=0.45\textwidth]{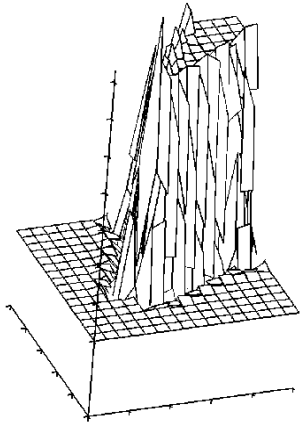} \label{fig:0430bdG} }\hspace{10pt}
  \subfigure[The cdG approximation with $\tol = 10^{-4}$.]{\includegraphics[width=0.45\textwidth]{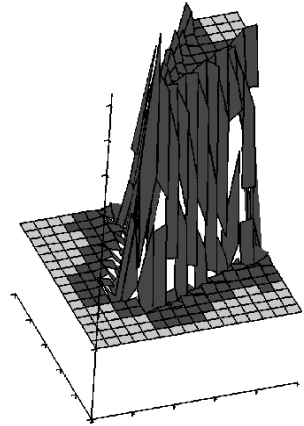} \label{fig:0430bcdG}}\\
  \subfigure[The cdG approximation with $\tol = 10^{-3}$.]{\includegraphics[width=0.45\textwidth]{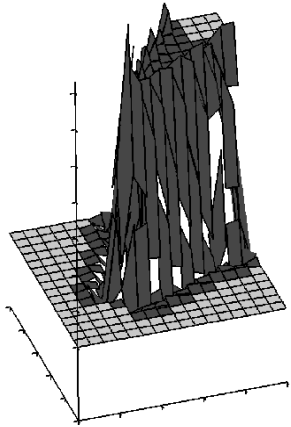} \label{fig:0430bcdG2}}\hspace{10pt}
  \subfigure[The fully continuous approximation.]{\includegraphics[width=0.45\textwidth]{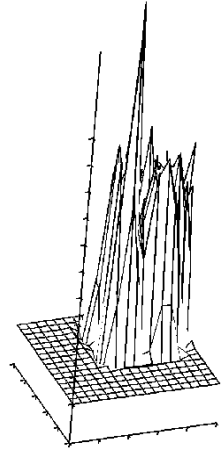} \label{fig:0430bcG}}
  \caption{A plot of the cG method, cdG method and dG method at time 1.52 (380 time steps). The discontinuous region is marked in dark grey for the cdG method. There is no appreciable difference between the first three plots. The oscillations are clearly visible in the fully continuous plot.}
  \label{fig:0430b}
\end{figure}

\clearpage


%


\begin{thebibliography}{10}

\bibitem{AM09}
{\sc B.~Ayuso and L.~D. Marini}, {\em Discontinuous {Galerkin} methods for
  advection-diffusion-reaction problems}, SIAM Journal on Numerical Analysis,
  47 (2009), pp.~1391--1420.

\bibitem{BHK07}
{\sc W.~Bangerth, R.~Hartmann, and G.~Kanschat}, {\em {deal.II} -- a general
  purpose object oriented finite element library}, ACM Transactions on
  Mathematical Software, 33 (2007), pp.~24/1--24/27.

\bibitem{DealIIRef}
{\sc W.~Bangerth and G.~Kanschat}, {\em {\tt deal.{I}{I}} Differential
  Equations Analysis Library, Technical Reference}.
\newblock \texttt{http://www.dealii.org}.

\bibitem{BJM09}
{\sc S.~Bartels, M.~Jensen, and R.~M\"{u}ller}, {\em Discontinuous {Galerkin}
  finite element convergence for incompressible miscible displacement problems
  of low regularity}, SIAM J. Numer. Anal., 47 (2009), pp.~3720--3743.

\bibitem{Bas99}
{\sc P.~Bastian}, {\em Numerical Computation of Multiphase Flows in Porous
  Media}, {H}abilitationsschrift, Christian-Albrechts-Universit{{\"a}}t Kiel,
  1999.

\bibitem{Bea88}
{\sc J.~Bear}, {\em Dynamics of Fluids in Porous Media}, Dover, New York, 1988.

\bibitem{BMS04}
{\sc F.~Brezzi, L.~D. Marini, and E.~S{{\"u}}li}, {\em Discontinuous {Galerkin}
  methods for first-order hyperbolic problems}, Mathematical Models and Methods
  in Applied Sciences, 14 (2004), pp.~1893--1903.

\bibitem{BHS06}
{\sc A.~Buffa, T.~J.~R. Hughes, and G.~Sangalli}, {\em Analysis of a multiscale
  discontinuous {Galerkin} method for convection-diffusion problems}, SIAM
  Journal on Numerical Analysis, 44 (2006), pp.~1420--1440.

\bibitem{BQS10}
{\sc E.~Burman, A.~Quarteroni, and B.~Stamm}, {\em Interior penalty continuous
  and discontinuous finite element approximations of hyperbolic equations},
  Journal of Scientific Computing, 43 (2010), pp.~293--312.
\newblock 10.1007/s10915-008-9232-6.

\bibitem{CGJ06}
{\sc A.~Cangiani, E.~H. Georgoulis, and M.~Jensen}, {\em Continuous and
  discontinuous finite element methods for convection-diffusion problems: A
  comparison}, in International Conference on Boundary and Interior Layers,
  G\"{o}ttingen, July 2006.

\bibitem{CJ86}
{\sc G.~Chavent and J.~Jaffre}, {\em Mathematical Models and Finite Elements
  for Reservoir Simulation Single Phase, Multiphase and Multicomponent Flows
  through Porous Media}, vol.~17 of Studies in Mathematics and Its
  Applications, Elsevier, 1986.

\bibitem{CE99}
{\sc Z.~Chen and R.~Ewing}, {\em Mathematical analysis for reservoir models},
  SIAM Journal of Mathematical Analysis, 30 (1999), pp.~431--453.

\bibitem{DP02}
{\sc C.~Dawson and J.~Proft}, {\em Coupling of continuous and discontinuous
  {G}alerkin methods for transport problems}, Computer Methods in Applied
  Mechanics and Engineering, 191 (2002), pp.~3213 -- 3231.

\bibitem{DFG07}
{\sc P.~R.~B. Devloo, T.~Forti, and S.~M. Gomes}, {\em A combined
  continuous-discontinuous finite element method for convection-diffusion
  problems}, Latin American Journal of Solids and Structures, 2 (2007),
  pp.~229--246.

\bibitem{Dur08}
{\sc R.~G. Dur\'{a}n}, {\em Mixed finite element methods}, in Mixed Finite
  Elements, Compatibility Conditions, and Applications, D.~Boffi and
  L.~Gastaldi, eds., vol.~1939 of Lecture Notes in Mathematics, Springer, 2008,
  pp.~1--44.

\bibitem{Fen95}
{\sc X.~B. Feng}, {\em On existence and uniqueness results for a coupled system
  modeling miscible displacement in porous media}, Journal of Mathematical
  Analysis and Applications, 194 (1995), pp.~883 -- 910.

\bibitem{HSS02}
{\sc P.~Houston, C.~Schwab, and E.~S\"{u}li}, {\em Discontinuous hp-finite
  element methods for advection-diffusion-reaction problems}, SIAM Journal on
  Numerical Analysis, 39 (2002), pp.~2133--2163.

\bibitem{JP86}
{\sc C.~Johnson and J.~Pitk{\"a}ranta}, {\em An analysis of the discontinuous
  {G}alerkin method for a scalar hyperbolic equation}, Math. Comp., 46 (1986),
  pp.~1--26.

\bibitem{LN00}
{\sc M.~G. Larson and A.~J. Niklasson}, {\em Conservation properties for the
  continuous and discontinuous {G}alerkin methods}, Tech. Rep. 2000-08,
  Chalmers University of Technology, 2000.

\bibitem{RT77}
{\sc P.~Raviart and J.~Thomas}, {\em A mixed finite element method for second
  order elliptic problems}, in Mathematical Aspects of the Finite Element
  Method, I.~Galligani and E.~Magenes, eds., vol.~606 of Lecture Notes in
  Mathematics, Springer, 1977.

\bibitem{SRW02}
{\sc S.~Sun, B.~Rivi{{\`e}}re, and M.~F. Wheeler}, {\em A combined mixed finite
  element and discontinuous {G}alerkin method for miscible displacement problem
  in porous media.}, in Recent Progress in Computational and Applied PDEs, New
  York, 2002, Kluwer/Plenum, pp.~323--351.

\end{thebibliography}
\end{document}